\documentclass[a4paper,12pt]{article}
\usepackage{latexsym}
\usepackage{amsmath}
\usepackage{amssymb}
\usepackage{enumerate}
\usepackage{enumitem}
\usepackage{tikz}
\usepackage{color}

\usepackage{theorem}
\newtheorem{theorem}{Theorem}[section]

\newtheorem{lemma}[theorem]{Lemma}

\theorembodyfont{\rmfamily}
\newtheorem{proof}{\textmd{\textit{Proof.}}}

\newtheorem{remark}[theorem]{Remark}

\newtheorem{acknowledgement}{\textmd{\textit{Acknowledgements.}}}

\makeatletter

\@addtoreset{equation}{section}
\makeatother

\newcommand{\qedd}{\hfill \Box}

\newcommand{\Sph}{\ensuremath{\mathbb{S}}}

\title{%A Maximal diameter sphere theorem for Berwald spaces %
Berwald spaces of bounded curvature are Riemannian
\footnote{
Mathematics Subject Classification (2010)\,:\,53C60, 53C22.}
}
\author{Nathaphon BOONNAM, Rattanasak HAMA, Sorin V. SABAU}
\date{}
\begin{document}

%%%%%%%%%%%%%%%%%%%%%%%%%%%%%%%%%%%%%%%%%%%%%%%%%%%

\maketitle
\begin{abstract}
	We prove that Berwald spaces whose flag curvature is nowhere vanishing are in fact Riemannian spaces. This means that any Berwald space with flag curvature bounded below by a positive number must be also Riemannian. This rigidity result shows the importance of non-Riemannian examples when imposing flag curvature bounds on Finsler spaces. 
\end{abstract}

\section{Introduction}

\quad Finsler manifolds are natural generalizations of the Riemannian ones in the sense that the metric depends not only on the point, but also on the direction as well. Even though classical Finsler geometry was mainly concerned with the local aspects of the theory, recently a lot of efforts was made to obtain global results in the geometry of Finsler manifolds (\cite{BCS}, \cite{KOT}, \cite{Oh}, \cite{Sa}, \cite{ST} and many others). 

In this quest of extending classical global results from Riemannian manifolds to Finslerian ones, the complicated form of the second variation formula of the arc length (see \cite{BCS}, \cite{S}), the non-reversibility of the geodesics or the non-symmetry of the distance function can be extremely inconvenient. 

A special class of Finsler manifolds are the so-called Berwald manifolds (or Berwald spaces) characterized by the property that its Chern connection coefficients depend only on the point and not on direction (see for instance \cite{BCS}). %it admits a torsion free linear connection which is compatible with the Finsler metric. 
There are many reasons to be interested in such special Finsler  structures. They are natural generalizations of Riemannian and locally Minkowski structures, and it happens that the Chern connection of a Berwald manifold is in fact the Levi-Civita connection of an Riemannian metric on $M$. Moreover, for a Berwald structure, all its tangent spaces are linearly isometric to a common Minkowski space (see Section 2 for details).   

Beside these elementary properties, Berwald manifolds have remarkable metric properties such that the vanishing of the tangent-curvature (see \cite{S}) and a very friendly form of the second variation arc-length formula (in fact similar to the Riemannian case) that allow extending different well-known comparison theorems from the Riemannian realm to the Finslerian setting. 

Our initial motivation in writing this paper was to extend the famous Maximal Diameter Sphere Theorem 
(see \cite{CE} or the original paper \cite{T}) 
 to Finsler manifolds, and the first natural attempt is to do it by  considering Berwald manifolds due to the reasons explained above. Obviously there are many other beautiful results in global differential geometry that need to be extended to the Finsler setting.

However, we found out that there is a big price to pay for trying to  extend formally such Riemmanian results, to the case of Berwald manifolds, by using upper or lower bounds on the flag curvature. In fact, these Berwald manifolds turn out to be Riemannian, so the extension of the results, without being wrong, are in fact trivial.

%-----------------------------------------

Here is our Main Theorem.

\begin{theorem}\label{thm: main_thm}
A complete Berwald manifold with nowhere vanishing flag curvature must be Riemannian. 
\end{theorem}

Here is the structure of this paper. In Preliminaries we recall the basic notions in Finsler geometry. In Section \ref{sec: Berwald} we recall the fundamental tools for proving our main Theorem \ref{thm: main_thm}, namely Szabo's classification of Berwald spaces (Lemma \ref{lem: Sz}) and Deng-Hou result on Berwald locally symmetric space (Lemma \ref{lem: DH}). Here is where we prove on Theorem \ref{thm: main_thm}.

Remark \ref{rem: pseudo-maximal} contains a pseudo-Maximal diameter sphere theorem for Berwald space that results directly by combining the Theorem \ref{thm: main_thm} with the famous Toponogov Theorem \ref{thm: Toponogov}. The statement looks as an interesting result, but a closer look shows that the bounded curvature hypothesis force the Finsler metric to be Riemannian, and then we simply quoted the Riemannian result. This is one type of frequent errors that should be avoided.

In a future research we will consider the extension to Berwald space of the Toponogov-type Maximal diameter sphere theorem in the case when the radius flag curvature  of $(M,F)$ is bounded below by the radial curvature of a 2-sphere of revolution with known cut locus as in (\cite{B}).

\acknowledgement{We express our gratitude to Prof. M. Tanaka for this problem. We also thank to Prof. S. Deng and V. Matveev for sharing with us their research, as well as Prof. H. Shimada and D. Bao for many useful suggestions.
}

%%%%%%%%%%%%%%%%%%%%%%%%%%%%%%%%%%%%%%%%%%%%%%%%%%%%
\section{Preliminaries}

\quad Let $M$ be an $n$-dimensional differentiable manifold with the local coordinate system $x=(x^1,\dots,x^n)$, and let $\pi:TM\to M$ be its tangent bundle with the natural projection $\pi$, and canonical coordinates $(x,y)$.  

Let us recall (\cite{BCS}) that  a {\it Riemann-Finsler manifold} $(M,F)$ is an $n$-dimensional differential manifold $M$ endowed with a norm 
$F:TM\to [0,\infty)$ such that
\begin{enumerate}
\item $F$ is positive and differentiable;
\item $F$ is 1-positive homogeneous, i.e. $F(x,\lambda y)=\lambda F(x,y)$, $\lambda>0$, $(x,y)\in TM$;
\item the Hessian matrix $g_{ij}(x,y):=\dfrac{1}{2}\dfrac{\partial^2 F^2}{\partial y^i\partial y^j}$ is positive definite on 
$\widetilde{TM}:=TM\setminus\{0\}$.
\end{enumerate}

The Riemann-Finsler structure is called {\it absolute homogeneous} if $F(x,-y)=F(x,y)$ because this leads to the homogeneity condition $F(x,\lambda y)=|\lambda| F(x,y)$, for any $\lambda\in \mathbb R$. 

By means of the Finsler fundamental function $F$ one defines the {\it indicatrix bundle} (or the Finslerian {\it unit sphere bundle}) by $SM:=\bigcup_{x\in M}S_xM$, where $S_xM:=\{y\in M\ : \ F(x,y)=1\}$.

In order to remove the redundancy resulting from the homogeneity, it is useful to consider all geometrical objects in Finsler geometry to be defined on the pull-back bundle $\pi^*TM$ over $SM$. Using the global section $l:=\frac{y^i}{F(x,y)}\frac{\partial}{\partial x^i}$ of $\pi^*TM$ one can construct a positively oriented $g$-orthonormal frame $\{e_1, e_2, \dots, e_n\}$ by putting $e_n:=l$, where $g=g_{ij}dx^i\otimes dx^j$ is the induced Riemannian metric on the fibres of $\pi^*TM$.  

The vector bundle $\pi^*TM$ has a torsion-free and almost $g$-compatible connection called {\it the Chern connection}
$D:C^\infty(TSM)\otimes C^\infty(\pi^*TM)\to C^\infty(\pi^*TM)$ given by 
\begin{equation*}
D_{\hat X}Z:=\{\hat X(z^i)+z^j\omega_j^{\ i}(\hat X)\}e_i,
\end{equation*}
where $\hat X$ is a vector field on $SM$, $Z:=z^ie_i$ is a section of $\pi^*TM$, and $\{\omega_j^{\ i}\}$ are the connection 1-forms on $SM$ (see \cite{BCS} or \cite{S} for details).

\begin{remark}
Locally, the connection 1-forms of the Chern connection are given by $\omega_j^{\ i}=\Gamma^i_{\ jk}(x,y)dx^k$, where 
\begin{equation*}
\Gamma^i_{\ jk}=\frac{1}{2}g^{is}\Bigl(\frac{\delta g_{js}}{\delta x^k}+  \frac{\delta g_{ks}}{\delta x^j}-
\frac{\delta g_{jk}}{\delta x^s} \Bigr)
\end{equation*}
are the local coefficients of the Chern connection, and $\Bigl(\frac{\delta }{\delta x^i}, F\frac{\partial }{\partial y^i}\Bigr)$ is a local adapted basis for $T(\widetilde{TM})$ determined by the non-linear connection of $(M,F)$ (see \cite{BCS} for details). 

\end{remark}

The curvature 2-forms of the Chern connection $\Omega_j^{\ i}:=d\omega_j^{\ i}-\omega_j^{\ k}\wedge\omega_k^{\ i}$ can be locally written as 
\begin{equation*}
\Omega_j^{\ i}=\frac{1}{2}R^{\ i}_{j\ kl}dx^k\wedge dx^l+P^{\ i}_{j\ kl}dx^k\wedge \frac{\delta y^l}{F}.
\end{equation*}

The concrete form of the coefficients $R^{\ i}_{j\ kl}$ and $P^{\ i}_{j\ kl}$ are not essential for our study and can be found in \cite{BCS}, \cite{S} or any other classical textbooks in Finsler geometry. 

The Finslerian analogue of the sectional curvature is the flag curvature defined as following (see \cite{BCS}). A {\it flag} on $M$ at a point $x\in M$ is given by $(y, \sigma)$, where the flag pole $y$ is a non-zero vector in $\sigma$, and 
the flag cloth $\sigma$ is a 2-plane in the tangent space $T_xM$. The {\it flag curvature} of the flag $(y,\sigma)$ is defined by
\begin{equation*}
K(y,\sigma):=\frac{V^i(y^jR_{jikl}y^l)V^k}{g(y,y)g(V,V)-[g(y,V)]^2},
\end{equation*}
where $V:=V^i\frac{\partial}{\partial v^i}$ is any non-zero vector in $\sigma$ such that $y$ and $V$ are linear independent.

 Let $\gamma:[a,b]\to M$ be a smooth piecewise curve on $M$ with the velocity $T(t):=\frac{d\gamma(t)}{dt}$, and let $W(t)$ be an arbitrary vector field along $\gamma$. The {\it (non-linear) covariant derivative} of $W$ along $\gamma$ is defined by
 \begin{equation*}
 D_TW:=\Bigl[ \frac{dW^i}{dt}+W^jT^k\Gamma^i_{\ jk}(\gamma(t),T(t))  \Bigr]\frac{\partial}{\partial x^i}|_{\gamma(t)},
 \end{equation*}
 where $\Gamma^i_{\ jk}$ are the coefficients of the Chern connection. The vector field $W$ is said to be {\it parallel} along $\gamma$ if $D_TW=0$. 
 
 This covariant derivative is used to define the {\it parallel translation} along $\gamma$ by 
 \begin{equation*}
 \mathcal P_\gamma:T_{\gamma(a)}M\to T_{\gamma(b)}M,\qquad \mathcal P_\gamma(v)=W(b),
 \end{equation*}
 where $W=W(t)$ is the parallel vector field along $\gamma$ with $W(a)=v$. Observe that for given initial conditions, this $W$ is unique. It is known that $\mathcal P_\gamma$ is a diffeomorphism that preserves the Finslerian norm. 
 
 Using the notion of parallel transport one can define the holonomy group $H_p$ of $(M,F)$ at a point $p$ by analogy with the Riemannian case (see for instance \cite{Ber}, \cite{CS}).  
 
 For any point $p\in  M$,  the holonomy group acts on the tangent space $T_pM$. This action may either be irreducible as a group representation, or reducible in the sense that there is a splitting of $T_pM$ into orthogonal subspaces $T_pM = T'_pM \oplus T''_pM$, each of which is invariant under the action of the holonomy group.
Depending on the holonomy group  $M$ is said to be irreducible or reducible, respectively.
 
 The theory of holonomy groups on Finsler manifolds is well developed (for details see for instance \cite{K}).%, but we omit here details since we are mainly concerned in the present paper with the Riemannian holonomy. 
 
 We also recall that a Finsler space $(M,F)$ is called a {\it locally symmetric Finsler space} if for any $p\in M$, there exists a neighbourhood $U_p$ of $p$ such that the geodesic symmetry with respect to $p$ is a local isometry of $U_p$. 
 \begin{remark}
 Such a Finsler space must be an absolute homogeneous one.
 \end{remark}
 
  A related notion is a {\it globally symmetric Finsler space}, that is a Finsler space whose each point is the isolated fixed point of an involutive isometry (see \cite{DH} for properties of local and global symmetric Finsler spaces).
 
Riemannian manifolds are trivial examples of Finsler structures. The Levi-Civita connection coincides with the Chern connection and the sectional curvature to the flag curvature of the Finsler counterpart. 

The simplest non-trivial examples are the locally Minkowski spaces. We recall that a Finsler manifold $(M,F)$ is called a {\it locally Minkowski space} if there exist certain privileged local coordinates $x=(x^i)$ on $M$, which together with the canonical coordinates $(x,y)$ induced in $TM$ make the fundamental function $F$ dependent only on $y$ and not on $x$. It is easy to see that locally Minkowski spaces are characterized by the vanishing of the curvature coefficients $R$ and $P$ of the Chern connection, hence they must have zero flag curvature for all flags on $M$.

%%%%%%%%%%%%%%%%%%%%%%%%%%%%%%%%%%%%%%%%%%%%
\section{Berwald spaces}\label{sec: Berwald}

\quad Another simple example of Finsler manifolds are the so-called Berwald spaces that play a central role in the present paper. A Finsler manifold $(M,F)$ is called a {\it Berwald space} if the Chern connection coefficients $\Gamma^i_{\ jk}$ are constant on each punctured tangent space $\widetilde{T_xM}$ at each point $x\in M$. There is a variety of conditions characterizing Berwald spaces (see \cite{BCS}) and remarkable properties of these spaces, we recall here only the most important ones. 

%\begin{proposition}\rm{}
Let $(M,F)$ be a Berwald space and $\gamma:[a,b]\to M$ a smooth curve on $M$ with velocity vector $T(t)=\frac{d\gamma}{dt}$. Then the (non-linear) parallel translation $\mathcal P_\gamma:T_{\gamma(a)}M\to T_{\gamma(b)}M$ is a linear isomorphism (\cite{I}).
%\end{proposition}
Remark that in the assumption that $(M,F)$ is Berwald is essential. Indeed, in general, for arbitrary Finsler metrics the parallel translation is not linear.

It is easy to see that  the forward completeness of a Berwald space implies the backward completeness and hence, similar with the Riemannian case there is no distinction between forward and backward completeness. It is therefore natural to call Berwald spaces simply {\it complete}. Indeed, a countable product of complete metric spaces is a complete metric space as well. 

However, the most intriguing property of Berwald spaces is the existence of a Riemannian metric $g_F$, called the {\it Binet-Legendre metric} on $M$ whose Levi-Civita connection coincides with the Chern connection of the Berwald metric $(M,F)$ (see \cite{MT}, \cite{Sz}). This implies that the geodesics of the Berwald space $(M,F)$ coincide with the geodesics of $(M,g_F)$ and same for the holonomy  groups at a point $p\in M$. 

Using this last remarkable property the Berwald spaces can be classified as in the following

\begin{lemma}{\rm (\cite{Sz})}\label{lem: Sz}
A connected Berwald manifold $(M,F)$ must be one of the following four types:
\begin{enumerate}
\item[(1)] $(M,F)$ is a Riemannian manifold;
\item[(2)] $(M,F)$ is a locally Minkowski space;
\item[(3)] $(M,F)$  is locally irreducible and it is a locally symmetric non-Riemannian Berwald space of rank $r\geq 2$;
\item[(4)] $(M,F)$ has locally reducible, and in this case $M$ can be locally decomposed in a Descartes product of Riemannian spaces, locally Minkowski spaces and locally symmetric  non-Riemannian Berwald space of rank $r\geq 2$.
\end{enumerate}
\end{lemma}

Using this classification we are interested in learning more about Berwald spaces of bounded curvature below, but we can consider here the much more general case of non-vanishing flag curvature as in our main Theorem  \ref{thm: main_thm}. 

%We have the following Key Lemma.

%\begin{lemma}[Key Lemma] \label{thm: main_thm}
%A complete Berwald manifold with nowhere vanishing flag curvature must be Riemannian. 
%\end{lemma}

\begin{proof}[Proof of the Theorem \ref{thm: main_thm}]

Let us assume that the flag curvature of the Berwald space $(M,F)$ is nowhere vanishing on $M$ for any flag $(y,\sigma)$. We will show that $(M,F)$ must be Riemannian. We can therefore eliminate cases (1) and (2) because (1) is trivial and (2) must have flag curvature zero that means it does not satisfy the hypothesis of our theorem. 

Let us consider the case (3) when then holonomy group is irreducible and $(M,F)$ is a locally symmetric non-Riemannian Berwald space of rank $r\geq 2$. 

The argument is eased by the following result.

\begin{lemma}\rm{(\cite{DH})}\label{lem: DH}
{\it
Let $(M,F)$ be a complete Berwald locally symmetric space. If the flag curvature of $(M,F)$ is nowhere zero, then $F$ is Riemannian.
}
\end{lemma}
(Compare with  \cite{Mat}).

It follows immediately that the Berwald spaces of type (3) with nowhere vanishing flag curvature must be Riemannian and this case is also solved. 

The remaining case (4) is when the holonomy group is reducible and in this case the manifold $M$ can be written as the product of, say a Riemannian, a locally Minkowski and a locally symmetric (rank $r\geq 2$) space (the number of factors is not important). Recall that in this case the tangent space to $M$ is also obtained as the product of the tangent spaces to each factor. 

Since $(M,F)$ is a Berwald space, its Chern connection coincides to the Levi-Civita connection of a Riemannian space $(M,g_F)$ (the Binet-Legendre metric studied in \cite{MT}), as explained already, so we can think of the Chern connection in terms of a Riemannian one. 

It is a classical fact of Riemannian geometry that the Levi-Civita connection of the Riemannian product space $(M,g_F)$ must satisfy $\nabla _XY=0$, for any tangent vectors $X$ and $Y$ belonging to the tangent spaces of different factors in the decomposition of $M$. Since this Levi-Civita connection coincides with the Chern connection of the Berwald space $(M,F)$, same fact can be said about the Chern connection. 

Indeed, if $(M_1,g_1)$, $(M_2,g_2)$ are Riemannian manifolds with Levi-Civita connection $\nabla^1$ and $\nabla^2$, respectively, then the Levi-Civita connection of the product manifold $M:=M_1\times M_2$ with Riemannian metric $g:=g_1+g_2$ is given by
\begin{equation*}
\nabla_{Y_1+Y_2}(X_1+X_2)=\nabla^1_{Y_1}X_1+\nabla^2_{Y_2}X_2,
\end{equation*}
where $X_1,Y_1\in \chi(M_1)$, $X_2,Y_2\in \chi(M_2)$, that is $\nabla$ has the properties
\begin{equation*}
\nabla_{Y_1}X_1=\nabla_{Y_1}^1X_1,\quad \nabla_{Y_2}X_2=\nabla_{Y_2}^2X_2
\end{equation*} 
and $\nabla_{Y_i}X_j=\nabla_{X_j}Y_i=0$, for any $i\neq j$. Note that here we canonically identify vector fields on the factors $M_1$ and $M_2$ with the vector fields on the product manifold $M_1\times M_2$ by $X_i\equiv (X_i,0)$, $Y_i\equiv (0,Y_i)$ for $i\in\{1,2\}$.

If we consider the curvature operator
\begin{equation*}
R(X,Y)Z=\nabla_X\nabla_YZ-\nabla_Y\nabla_XZ-\nabla_{[X,Y]}Z
\end{equation*}
on $M=M_1\times M_2$ it is easy to see that $R(X,Y)Z=0$, when $X\equiv (X_1,0)$, $Y\equiv(0,Y_2)$ come from different factors.

Moreover, for any section $\sigma=\text{span}\{X,Y\}\subset T_{(p,q)}(M_1\times M_2)$ such that $X\in T_pM_1$, $Y\in T_qM_2$, one can easily see that the sectional curvature of this $\sigma$ with respect to the product Riemannian structure vanishes, i.e. $K_{p,q}(\sigma)=0$. (see for instance \cite{DoCarmo}, Chapter 6, Exercise 1 or \cite{Peter}, Chapter 2, Exercise 22).

The geometrical intuition underlying this phenomena become clear if we think to the product of two spheres $M:=\Sph_a^n\times \Sph_b^m$ of dimensions $n$ and $m$ and radii $a$ and $b$, respectively. Straightforward computations show that all sectional curvatures lie in the interval $[0,\max\{a,b\}]$ and hence $\Sph_a^n\times \Sph_b^m$ has always constant scalar curvature, but can never have constant sectional curvature (see \cite{Peter}, Section 3.2)

This means that there must exist a tangent plane $\sigma$ to $M$ whose flag curvature, with respect to any flag pole $y$ is zero, that is $K(y,\sigma)$ cannot be non-vanishing for any point of $M$ and any flag. This proves our theorem. 
$\qedd$
\end{proof}
%%%%%%%%%%%%%%%%%%%%%%%%%%%%%%%
%\section{A maximal diameter sphere theorem for Berwald spaces}

\begin{remark}\label{rem: pseudo-maximal}
A classical result in Riemannian geometry due to Bonnet says that if the sectional curvature $K$ of a complete Riemannian manifold $(M,g)$ satisfies $K\geq H$, where $H$ is a positive constant, then $diam (M,g)\leq \frac{\pi}{\sqrt{H}}$ (see for instance \cite{Ber}). 

It is natural to ask what happens when the diameter of the Riemannian manifold is maximal, that is $diam (M,g)=\frac{\pi}{\sqrt{H}}$, and this led Toponogov to his famous Maximal diameter sphere theorem.

\begin{theorem}{\rm{(\cite{T})}}\label{thm: Toponogov}
Let $(M,g)$ be a complete connected $n$-dimensional Riemannian manifold such that 
\begin{enumerate}
\item[(1)] the sectional curvature is bounded from bellow by a positive constant $H$;
\item[(2)] $diam (M,g)=\frac{\pi}{\sqrt{H}}$.
\end{enumerate}
Then $(M,g)$ is isometric to the $n$-dimensional sphere of radius $\frac{\pi}{\sqrt{H}}$.
\end{theorem}

In the light of the previous section, one could formulate a pseudo-Maximal diameter sphere theorem for Berwald spaces.

Let $(M,F)$ be a complete connected Berwald space such that 
\begin{enumerate}
\item[(1)] the flag curvature is bounded below by a positive number, that is $K(y,\sigma)\geq H$ for any flag $(y,\sigma)$ at $x\in M$, where $H$ is a positive constant;
\item[(2)] $diam (M,F)=\frac{\pi}{\sqrt{H}}$. 

Then $(M,F)$ is isometric to the $n$-dimensional sphere of radius $\frac{\pi}{\sqrt{H}}$ with canonical Riemannian metric. 
\end{enumerate}

Indeed, let us assume the same conditions as Toponogov for a complete connected Berwald space $(M,F)$. Then flag curvature bounded condition combined with our Theorem \ref{thm: main_thm} shows that actually $(M,F)$ must be Riemannian, and hence the conclusion follows from Toponogov's Theorem  \ref{thm: Toponogov}.

However,  one should be aware of the fact that the hypothesis here implies that the Finsler metric is a Riemannian metric one and then we simply write down the Riemannian result. In order to around such traps are should always look for explicit non-trivial examples of Finsler metrics with the desired properties.

\end{remark}

%%%%%%%%%%%%%%%%%%%%%%%%%%%%%%%%%%%%%%%%%%%%
%%%%%%%%%%%%%%%%%%%%%%%%%

%------------------------------------

\bigskip

%Sorin V. SABAU

%\bigskip

\noindent
Nathaphon BOONNAM\\
Department of Applied Mathematics and Informatics\\
Faculty of Science and Industrial Technology, Prince of Songkla University\\
Muang Suratthani, Suratthani 84000 Thailand

%\medskip
{\tt nathaphon.b@psu.ac.th
}

\medskip

\noindent
 Rattanasak HAMA\\
Department of Mathematics,\\
Faculty of Science, KMITL, Bangkok, 10520, Thailand
%1 Chalongkrung 1 road , Lat Krabang, Lat Krabang, 

\medskip
{\tt rattanasakhama@gmail.com}

\medskip

\noindent 
Sorin V. SABAU\\
School of Biological Sciences,
Department of Biology,\\
Tokai University,
Sapporo 005\,--\,8600,
Japan

\medskip
{\tt sorin@tokai.ac.jp}

\end{document}